\documentclass{amsart}
\font\ehsc=cmcsc10 scaled 850

\let\<=\langle
\let\>=\rangle
\let\\=\backslash
\let\what=\widehat
\let\wtil=\widetilde
\let\sse=\subseteq
\let\noi=\noindent
\let\vphi=\varphi
\let\veps=\varepsilon
\let\limply=\Longrightarrow

\def\0{\{0\}}
\def\span{{\kern.5pt{\rm span}\kern1pt}}
\def\smallfrac#1#2{{\textstyle{\frac{#1}{#2}}}}
\def\conv{{\;\longrightarrow\;}}
\def\wconv{{{\buildrel_{\scriptstyle w}\over\conv}}}
\def\kconv{{\conv_{_{\kern-13.5pt\scriptstyle k\kern9pt}}}}
\def\jwconv{{\wconv_{_{\kern-13.5pt\scriptstyle j\kern9pt}}}}

\def\emap{\hbox to25pt{\rightarrowfill}}

\def\nmap{\Big\uparrow}
\def\dsse{{\kern7pt{\overline{\phantom{\overline:}}{\kern-7pt\sse}}\kern3pt}}

\def\A{{\mathcal A}}
\def\B{{\mathcal B}}
\def\C{{\kern.5pt\mathcal C}}

\def\Oe{{\mathcal O}}
\def\R{{\mathcal R}}
\def\X{{\mathcal X}}
\def\Z{{\mathcal Z}}

\def\BX{{\B[\X]}}

\def\CC{{\mathbb C\kern.5pt}}
\def\DD{{\mathbb D\kern.5pt}}
\def\RR{{\mathbb R\kern.5pt}}
\def\TT{{\mathbb T\kern.5pt}}

\newsymbol\varnothing 203F
\let\void=\varnothing

\def\matrix#1{\null\,\vcenter{
              \normalbaselines\mathsurround=0pt\ialign{
              \hfil $##$
              \hfil && \quad
              \hfil $##$
              \hfil \crcr
              \mathstrut \crcr
              \noalign{\kern-\baselineskip}#1 \crcr
              \mathstrut \crcr
              \noalign{\kern-\baselineskip} \crcr }}\,}

\newtheorem{theorem}{Theorem}

\newtheorem{corollary}{Corollary}

\theoremstyle{definition}

\newtheorem{remark}{Remark}

\numberwithin{theorem}{section}
\numberwithin{lemma}{section}
\numberwithin{corollary}{section}
\numberwithin{proposition}{section}
\numberwithin{conjecture}{section}
\numberwithin{definition}{section}
\numberwithin{remark}{section}
\numberwithin{question}{section}

\begin{document}

\vglue-65pt\noindent
\hfill{\it Advances in Mathematical Sciences and Applications}\/
{\bf 30} (2021) 571-585

\vglue20pt
\title{Weakly Supercyclic Power Bounded Operators of Class {\bf C}$_1$.}
\author{C.S. Kubrusly}
\address{Mathematics Institute, Federal University of Rio de Janeiro, Brazil}
\email{carloskubrusly@gmail.com}
\author{B.P. Duggal}
\address{Faculty of Sciences and Mathematics, University of Ni\v s, Serbia}
\email{bpduggal@yahoo.co.uk}
\renewcommand{\keywordsname}{Keywords}
\date{August 15, 2021}

\begin{abstract}
There is no supercyclic power bounded operator of class
$C_{1{\textstyle\cdot}}.$ There exist, however, weakly l-sequentially
supercyclic unitary operators$.$ We show that if $T$ is a weakly
l-sequentially supercyclic power bounded operator of class
$C_{1{\textstyle\cdot}}$, then it has an extension $\what T$ which is
a weakly l-sequentially supercyclic singular-continuous unitary (and
$\what T$ has a Rajchman scalar spectral measure whenever $T$ is weakly
stable)$.$ The above result implies
$\sigma_{\kern-1ptP}(T)=\sigma_{\kern-1ptP}(T^*)=\void$, and also that
if a weakly l-sequentially supercyclic operator is similar to an isometry,
then it is similar to a unitary operator.
\end{abstract}

\maketitle

\vskip-15pt\noi
\section{Introduction}

The motivation for the present paper is synthesized in the following chain of
results and questions$.$ (Some of these hold in a normed space as will
become clear from the text but all of them certainly hold for Hilbert-space
operators --- notation and terminology will be described in the next section.)
\vskip2pt\noi
\begin{description}
\item{$\kern-7pt$\rm(A)$\kern2pt$}
{\it There is no supercyclic isometry}\/ $\;\;$(on a complex Banach space)
\vskip2pt\noi
\end{description}
\cite[proof of Theorem 2.1]{AB} (also \cite[Lemma 4.1]{KD1}).
\vskip2pt\noi
\begin{description}
\item{$\kern-7pt$\rm(B)$\kern2pt$}
{\it A power bounded operator of class\/ $C_{1{\textstyle\cdot}\!}$ is not
supercyclic}
\vskip2pt\noi
\end{description}
\cite[Theorem 2.1]{AB} (recall$:$ isometries are
$C_{1{\textstyle\cdot}}$-contractions).
\vskip2pt\noi
\begin{description}
\item{$\kern-7pt$\rm(C)$\kern2pt$}
{\it A supercyclic power bounded operator is strongly stable}\/.
\end{description}
\vskip2pt\noi
This is a (nontrivial) consequence of (B) proved in \cite[Theorem 2.2]{AB}
(see extension in \cite[Section 7]{Ker3})$.$ Such a result suggests the
following weak counterpart$.$ Question$:$
\vskip2pt\noi
\begin{description}
\item{$\kern-7pt$\rm(D)$\kern2pt$}
{\it is a weakly l-sequentially supercyclic power bounded operator weakly
stable}$\kern1pt?$
\end{description}
\vskip2pt\noi
The question was raised in \cite{KD1} and remains unanswered$.$ We will show
here that
\vskip2pt\noi
\begin{description}
\item{$\kern-7pt$\rm(E)$\kern2pt$}
{\it if\/ $T$ is a weakly l-sequentially supercyclic power bounded of class\/
$C_{1{\textstyle\cdot}\!}$, then}
$$
\sigma_{\kern-1ptP}(T)=\sigma_{\kern-1ptP}(T^*)=\void.
$$
\vskip-6pt\noi
\end{description}
\vskip2pt\noi
Another related open question:
\vskip2pt\noi
\begin{description}
\item{$\kern-7pt$\rm(F)$\kern2pt$}
{\it is a weakly l-sequentially supercyclic power bounded operator of class\/
$C_{1{\textstyle\cdot}\!}$ similar to a unitary operator}$\kern1pt?$
\end{description}
\vskip2pt\noi
In fact (see \cite[Example 3.6]{BM1}, \cite[pp.10,12]{BM1},
\cite[Proposition 1.1, Theorem 1.2]{Shk}),
\vskip2pt\noi
\begin{description}
\item{$\kern-7pt$\rm(G)$\kern2pt$}
{\it there exist weakly l-sequentially supercyclic unitary operators}\/.
\end{description}
\vskip2pt\noi
Moreover \cite[Theorem 4.2]{Kub1},
\vskip2pt\noi
\begin{description}
\item{$\kern-7pt$\rm(H)$\kern2pt$}
{\it every weakly l-sequentially supercyclic unitary operator is
singular-continuous}\/,
\end{description}
\vskip2pt\noi
and (see, e.g., \cite[Propositions 3.2 and 3.3]{Kub1}),
\vskip2pt\noi
\begin{description}
\item{$\kern-6pt$\rm(I)$\kern3pt$}
{\it there are weakly stable and also weakly unstable singular-continuous
unitary operators}\/.
\end{description}

\vskip6pt 
Such a chain of statements and questions helps to assemble a string of
arguments to approach the main results of the present paper, which are
described in the next paragraph$.$ Question (F) is linked to a well-known
result$:$ {\it A weakly l-sequentially supercyclic hyponormal operator}\/
(in particular, {\it a weakly l-sequentially supercyclic isometry\/$)$ is a
multiple of a unitary operator}\/ (and so {\it it is similar to a unitary
opera\-tor}\/) \cite[Theorem 3.4]{BM1}, which has been extended beyond
hyponormal operators in \cite[Corollary 3.1]{Dug}, \cite[Theorem 2.7]{DKK}$.$
This complements the result which says$\kern.5pt:$ {\it no hypo\-normal
operator}\/ (in particular, {\it no isometry\/$)$ is supercyclic}\/
\cite[Theorem 3.1]{Bou} (for the particular case,
\cite[proof of Theorem 2.1]{AB}, \cite[Lemma 4.1]{KD1})$.$ Question (F) is
linked to a classical result as well$:$ {\it a Hilbert-space operator is an
invertible power bounded with a power bounded inverse if and only if it is
similar to a unitary operator}\/ \cite{Nag}.

\vskip6pt
Power bounded operators of class $C_{1{\textstyle\cdot}}$ together with some
notion of cyclicity have played a significant role in operator theory (for
instance, in connection with the invariant subspace problem --- see, e.g.,
\cite{Ker0,Kub0})$.$ The original contribution to linear dynamics in this paper
is the characterization of weak l-sequential supercyclicity for power bounded
operators of class $C_{1{\textstyle\cdot}\!}$ as a nontrivial extension of (B)
from \cite[Theorem 2.1]{AB}$.$ The main results along this line are presented
in Theorems 3.1 and 3.2, leading to the relevant point spectra identity stated
in item (E) above, which is proved in Corollary 4.1$.$ Corollary 4.2 shows
that {\it a weakly l-sequentially supercyclic operator similar to an isometry
is indeed similar to a unitary operator}\/.

\section{Notation and Terminology}

Throughout this paper $\X$ stands for a complex normed space (the special
cases of inner product, Banach, and Hilbert spaces are discussed accordingly),
and $\BX$ stands for the normed algebra of all operators on $\X\kern-1pt$
(i.e., of all boundedu linear transformations of $\X$ into itself)$.$ The
linear manifold ${\R(T)=T(\X)}$ of $\X$ is the range of $T$ in $\BX.$ Let
${\X^*\!=\B[\X,\CC]}$ be the dual of $\X$ and let $T^*\kern-1pt$ in
$\B[\X^*]$ be the normed-space adjoint of $T$ (same notation for the
Hilbert-space adjoint of operators on a Hilbert space where the concepts of
dual and adjoint are shaped by the Riesz Representation Theorem)$.$ An
operator ${T\in\BX}$ is power bounded if ${\sup_n\kern-1pt\|T^n\|<\infty}$
(i.e., if ${\sup_n\kern-1pt\|T^nx\|\kern-1pt<\kern-1pt\infty}$ for every
${x\kern-1pt\in\kern-1pt\X}$ by the Banach--Steinhaus Theorem if $\X$ is
Banach)$.$ An operator $T$ is strongly or weakly stable if the $\X$-valued
power sequence $\{T^nx\}$ converges strongly or weakly to zero for every
${x\in\X}$ (i.e., if ${T^nx\to0}$ which means ${\|T^nx\|\to0}$, or
${T^nx\wconv0}$ which means ${f(T^nx)\to0}$ for every
${f\kern-1.5pt\in\kern-1pt\X^*}\!$, for every ${x\kern-1pt\in\kern-1pt\X}$,
respectively)$.$ It is uniformly stable if the $\BX$-valued sequence $\{T^n\}$
converges (in the uniform operator topology) to the null operator$.$ Thus
uniform stability implies strong stability, which implies weak stability,
which in turn implies power bounded\-ness$.$ An operator $T$ is of class
$C_{0{\,\kern-1pt\textstyle\cdot}}$ if it is strongly stable, and of class
$C_{{\textstyle\cdot}0}$ if its adjoint $T^*\kern-1pt$ is strongly stable$.$
It is of class $C_{1{\textstyle\cdot}}$ if ${T^nx\not\to0}$ for every nonzero
${x\kern-1pt\in\kern-1pt\X}$, and of class $C_{{\textstyle\cdot}1}$ if\/
${T^{*n}f\not\to0}$ for every nonzero ${f\kern-1pt\in\kern-1pt\X^*\!}$ (or
${T^{*n}x\not\to0}$ for every nonzero ${x\kern-1pt\in\kern-1pt\X}$ if $\X$ is
Hilbert)$.$ All combinations are possible leading to classes $C_{00}$,
$C_{01}$, $C_{10}$, $C_{11}$.

\vskip6pt
Let the orbit of a vector ${y\in\X}$ under an operator ${T\in\BX}$ be the set
$$
\Oe_T(y)=\big\{T^ny\in\X\!:\,n\ge0\big\}.
$$
The orbit of the one-dimensional space spanned by $y$,
$$
\Oe_T(\span\{y\})=\big\{\alpha T^ny\in\X\!:\,\alpha\in\CC,\,n\ge0\big\},
$$
is referred to as the projective orbit of a vector ${y\in\X}$ under an
operator ${T\in\BX}.$ An operator $T$ is {\it hypercyclic}\/ if the orbit of
sone vector ${y\in\X}$ is dense in $\X$, that is ${\Oe_T(y)^-\!=\X}$ where
the upper bar $^-$ stands for closure in the norm topology$.$ It is
{\it cyclic}\/ if ${\span\Oe_T(y)^-\!=\X}$$:$ the linear span of the orbit of
some $y$ is dense in $\X.$ A nonzero vector ${y\in\X}$ is a {\it supercyclic
vector}\/ for an operator ${T\in\BX}$ if the projective orbit of $y$ is dense
in $\X$ (in the norm topology), that is if
$$
\Oe_T(\span\{y\})^-\!=\X.
$$
Thus a nonzero ${y\in\X}$ is a supercyclic vector for $T$ if and only if
for every ${x\in\X}$ there exists a sequence of nonzero complex numbers
$\{\alpha_i\}_{i\ge0}$ (which depends on $x$ and $y$) such that for some
subsequence $\{T^{n_i}\}_{i\ge0}$ of $\{T^n\}_{n\ge0}$
$$
\alpha_iT^{n_i}y\to x
\qquad
\hbox{(i.e., $\|\alpha_iT^{n_i}y-x\|\to0$}).
$$
If ${T\in\BX}$ has a supercyclic vector, then it is a {\it supercyclic
operator}$.$ The weak counterpart of the above convergence criterion reads
as follows$.$ A nonzero vector ${y\in\X}$ is a {\it weakly l-sequentially
supercyclic vector}\/ for an operator ${T\in\BX}$ if for every ${x\in\X}$
there exists a sequence of nonzero complex numbers $\{\alpha_i\}_{i\ge0}$
such that
$$
\alpha_iT^{n_i}y\wconv x
\qquad
\hbox{(i.e., $f(\alpha_iT^{n_i}y-x)\to0$ for every $f\in\X^*$)}
$$
for some subsequence $\{T^{n_i}\}_{i\ge0}$ of $\{T^n\}_{n\ge0}.$ An operator
$T$ in $\BX$ is a {\it weakly l-sequentially supercyclic operator}\/ if it
has a weakly l-se\-quentially supercyclic vector$.$ It is {\it weakly
supercyclic}\/ if there is a vector ${y\in\X}$ (called a {\it weak supercyclic
vector}\/) for which the projective orbit $\Oe_T(\span\{y\})$ is dense in $\X$
in the weak topology$.$ These are related as follows (and the converses fail
--- \cite[pp.38,39]{Shk}, \cite[pp.259,260]{BM2}):
\vskip6pt\noi
\centerline{\ehsc supercyclic $\,\limply\,$ weakly l-sequentially supercyclic
$\,\limply\,$ weakly supercyclic.}
\vskip6pt

Any form of cyclicity implies the operator $T$ acts on a separable space, and
so separability for $\X$ is a consequence of any form of cyclicity$.$ Several
forms of weak supercyclicity, including weak l-sequential supercyclicity, have
recently been examined in
\cite{San1, San2, BM1, MS, Shk, DKK, Dug, Kub1, KD1, KD2}$.$ The notion of
weak l-sequential supercyclicity was introduced explictly in \cite{BCS} and
implicitly in \cite{BM1}, and investigated in \cite{Shk} where a terminology
similar to the one adopted here was introduced (we use the letter ``l'' for
``limit'' instead of the numeral ``1'' used in \cite{Shk} --- there are
reasons for both terminologies).

\vskip3pt\noi
\begin{remark}
If a vector $y$ in a normed space $\X$ is supercyclic, or weakly
l-se\-quentially supercyclic, or weakly supercyclic for $T$ on $\X$, then so
is any vector in $\Oe_T(\span\{y\})$ \cite[Lemma 5.1]{Kub2}, and so is in
particular the vector $Ty.$ Moreover,
$$
\Oe_T(\span\{Ty\})
=\big\{\alpha T^ny\in\X\!:\,\alpha\in\CC,\,n\ge1\big\}\sse\R(T)\sse\X.
$$
Thus if $y$ is weakly supercyclic for $T$ (in particular, if it is
l-sequentially supercyclic, or simply supercyclic) then its range $\R(T)$ is
weakly dense in $\X$, and so it is dense in the norm topology since $\R(T)$ is
convex in $\X$ (see, e.g., \cite[Theorem 2.5.16]{Meg})$.$ (Note: The above
argument gives still another the proof for \cite[Lemma 4.1]{KD2}.)
\end{remark}

\section{Weak l-Sequential Supercyclicity and Power Boundedness}

Banach limit is a standard tool in functional analysis, whose existence was
established by Banach himself \cite[p.21]{Ban} as a consequence of the
Hahn--Banach Theorem$.$ Let $\ell_+^\infty$ denote the Banach space of all
complex-valued bounded sequences equipped with its usual sup-norm$.$ A Banach
limit is a bounded linear functional ${\vphi\!:\ell_+^\infty\!\to\CC}$
assigning a complex number to each sequence ${\{\xi_n\}\in\ell_+^\infty}.$ For
existence and properties of Banach limits see, e.g., \cite[Section III.7]{Con}
or \cite[Problem 4.66]{EOT})$.$ The Banach limit technique used here has
originated in the celebrated Nagy's 1947 paper \cite{Nag}, and has been
applied quite often since then (see, e.g., \cite{Ker1, Ker2, Ker3} for
applications to power bounded operators which is the focus here)$.$ For
a recent survey see \cite{KD3}.

\vskip3pt\noi
\begin{theorem}
Let\/ $T$ be a power bounded operator of class\/ $C_{1{\textstyle\cdot}\!}$
on a normed space\/ ${(\X,\|\cdot\|)}.$ Then
\vskip4pt\noi
\begin{description}
\item{$\kern-8pt$\rm(a)$\kern2pt$}
there is a norm\/ ${\|\cdot\|_\vphi}$ on\/ $\X$ for which\/ $T$ is an
isometry on\/ ${(\X,\|\cdot\|_\vphi)}$.
\vskip4pt\noi
\item{$\kern-8pt$\rm(b)$\kern2pt$}
If a vector\/ ${y\in\X}$ is weakly l-sequentially supercyclic for\/ $T$ when
it acts on the normed space\/ ${(\X,\|\cdot\|)}$, then the same vector\/
${y\in\X}$ is weakly l-sequentially supercyclic for\/ $T$ when it acts on the
normed space\/ ${(\X,\|\cdot\|_\vphi)}$.
\vskip4pt\noi
\item{$\kern-8pt$\rm(c)$\kern2pt$}
If\/ $T$ is weakly l-sequentially supercyclic when acting on\/
${(\X,\|\cdot\|)}$, then it has an extension\/ $\what T$ on the
completion\/ ${(\what\X,\|\cdot\|_{\what\X})}$ of\/ ${(\X,\|\cdot\|_\vphi)}$
which is a weakly l-sequentially supercyclic isometric isomorphism\/.
\vskip4pt\noi
\item{$\kern-8pt$\rm(d)$\kern2pt$}
If\/ $T$ on\/ ${(\X,\|\cdot\|)}$ is weakly stable, then\/ $\what T$ on\/
${(\what\X,\|\cdot\|_{\what\X})}$ is weakly stable.
\end{description}
\end{theorem}

\begin{proof}
(a)
Let ${\vphi\!:\ell_+^\infty\!\to\CC}$ be a Banach limit$.$ Take the norm
${\|\cdot\|\!:\X\!\to\RR}$ on $\X.$ Suppose ${T\in\BX}$ is power
bounded$.$ Since ${\{\|T^nx\|\}\in\ell_+^\infty}$ set for each $x$ in $\X$
$$
\|x\|_\vphi\!=\vphi(\{\|T^nx\|\}).
$$
Since $T$ is power bounded, and since a Banach limit $\vphi$ is
order-preserving for real-valued bounded sequences (i.e.,
$\vphi(\{\xi_n\})\le\vphi(\{\upsilon_n\})$ if ${\xi_n\le\upsilon_n}$ in $\RR$
for every $n$) with $\vphi(\{1,1,1,\dots\})=1$, then for every ${x\in\X}$
$$
\|x\|_\vphi\!\le{\sup}_n\|T^n\|\,\|x\|.
$$
This defines a seminorm ${\|\cdot\|_\vphi\!:\X\!\to\RR}$ on $\X.$ However, as
the power bounded operator $T$ is of class $C_{1{\textstyle\cdot}}$, then
${0<\liminf_n\|T^nx\|}$ for ${x\ne0}.$ Since
$\liminf_n\xi_n\le\vphi(\{\xi_n\})$ for every real-valued bounded sequence and
any Banach limit $\vphi$, then $\|\cdot\|_\vphi$ becomes a norm on $\X.$
Consider the normed space ${(\X,\|\cdot\|_\vphi)}.$ For simplicity write $\X$
for the normed space ${(\X,\|\cdot\|)}$ and write $\X_\vphi$ for the normed
space ${(\X,\|\cdot\|_\vphi)}$ --- same underlying linear space $\X.$ The norm
$\|\cdot\|_\vphi$ makes the operator $T$ into an isometry when acting on
$\X_\vphi$ (i.e., ${T\in\B[\X_\vphi]}$ is an isometry)$.$ Indeed, as a
Banach limit $\vphi$ is backward-shift-invariant, then for every ${x\in\X}$
$$
\|Tx\|_\vphi=\vphi(\{\|(T^{n+1}x)\|\})
=\vphi(\{\|(T^nx)\|\})=\|x\|_\vphi.
$$

\vskip4pt\noi
(b)
Since ${T\in\BX}$ is power bounded and of class $C_{1{\textstyle\cdot}}$,
then by the previously dis\-played inequality the norms are related by
${\|\cdot\|_\vphi\le\beta\kern1pt\|\cdot\|}$ with ${\beta=\sup_n\|T^n\|}.$ So
$$
\sup_{x\ne0}\frac{|f(x)|}{\|x\|}
\le\beta\,\sup_{x\ne0}\frac{|f(x)|}{\|x\|_\vphi}
$$
for every linear functional ${f\!:\X\!\to\CC}.$ As $\X^*$ is the dual of $\X$,
let $\X^*_\vphi$ denote the dual of $\X_\vphi.$ By the above inequality if
${f\kern-2pt:\kern-1pt\X\kern-2pt\to\kern-1pt\CC}$ is bounded as a linear
functional on ${\X_\vphi\kern-1pt=(\X,\|\cdot\|_\vphi)}$, then it is bounded
as a linear functional on ${\X=(\X,\|\cdot\|)}.$ Thus
$$
\X^*_\vphi\sse\X^*.
$$
Therefore if ${x,y\in\X}$ and if ${f(\alpha_iT^{n_i}y-x)\to0}$ for every
${f\in\X^*}\!$ then, in particular, ${f(\alpha_iT^{n_i}y-x)\to0}$ for
every ${f\!\in\!\X^*_\vphi}$, proving item (b).

\vskip6pt\noi
(c)
Consider the completion $(\what\X,\|\cdot\|_{\what\X})$ of the normed space
$(\X,\|\cdot\|_\vphi).$ Write $\what\X$ for $(\what\X,\|\cdot\|_{\what\X})$
and so $\what\X$ is the completion of $\X_\vphi.$ Thus $\X_\vphi$ is densely
embedded in $\what\X$, which means there is an isometric isomorphism
${J\!:\X_\vphi\kern-1pt\to J(\X_\vphi)=\wtil\X}$ where $\wtil\X$ is dense in
$\what\X.$ Let ${\what T\kern-1pt\in\B[\what\X]}$ be the extension of
${T\kern-1pt\in\B[\X_\vphi]}$ so that its restriction to $\wtil\X$ is
$\what T|_{\wtil\X}={J\kern1ptTJ^{-1}\!\in\B[\wtil\X]}.$ Thus $\wtil\X$ is
$\what T$-invariant and
$T={J^{-1}\what T|_{\wtil\X}J}={J^{-1}\what TJ\in\B[\X_\vphi]}$ as in the
following commutative diagram (the symbol $\kern-1pt\dsse\kern-1pt$ means
densely included):
\vskip9pt\noi
$$
\matrix{
\kern7pt\X_\vphi\phantom{\int_|}\kern-5pt
& \kern-10pt\buildrel J\over\emap
& \kern-18pt\wtil\X=J(\X_\vphi)
& \kern-30pt\dsse
& \kern-25pt\what\X                                                   \cr
\kern-5pt\scriptstyle{T}\kern1pt\nmap
&
& \kern-7pt\scriptstyle{\nmap \what T|_{\wtil\X}\,=J\,TJ^{-1}}
&
& \kern-20pt\kern1pt\nmap\scriptstyle{\what T}                        \cr
\kern7pt\X_\vphi
& \kern-10pt\buildrel J\over\emap
& \kern-18pt \wtil\X=J(\X_\vphi)
& \kern-30pt\dsse
& \kern-25pt\;\what\X.                                                \cr}
$$
\vskip4pt

\vskip6pt\noi
{\it Claim 1}\/$.$
If there is a weakly l-sequentially supercyclic vector ${y\in\X_\vphi}$ for
${T\in\B[\X_\vphi]}$, then there is a weakly l-sequentially vector
${\wtil y\in\wtil\X\dsse\what\X}$ for ${\what T\in\B[\what \X]}$.

\vskip6pt\noi
{\it Proof}\/$.$
Consider the isometric isomorphism ${J\in\B[\X_\vphi,\wtil\X]}.$ Set
${\wtil T=\what T|_{\wtil\X}=J\kern1ptTJ^{-1}}$ in $\B[\wtil\X]$ so that
$\wtil T^n=(\what T|_{\wtil\X})^n=\what T^n|_{\wtil\X}=J\kern1ptT^nJ^{-1}$ for
every non\-negative integer $n$ (since $\wtil\X$ is $\what T$-invariant)$.$
We split the proof into two parts$.$

\vskip6pt\noi
(1) If there is a weakly l-sequentially supercyclic vector ${y\in\X_\vphi}$
for ${T\in\B[\X_\vphi]}$, then there is a weakly l-sequentially supercyclic
vector ${\wtil y\in\wtil\X}$ for ${\wtil T\in\B[\wtil\X]}$.

\vskip6pt\noi
(2) If there is a weakly l-sequentially supercyclic vector
${\wtil y\in\wtil\X}$ for ${\wtil T\in\B[\wtil\X]}$, then $\wtil y$ is a
weakly l-sequentially supercyclic vector for ${\what T\in\B[\what \X]}$.

\vskip6pt\noi
{\it Proof of\/ \rm(1)}\/$.$
If there exists a weakly l-sequentially supercyclic vector ${y\in\X_\vphi}$
for ${T\in\B[\X_\vphi]}$, then for each ${x\in\X_\vphi}$ there is a
sequence of nonzero numbers $\{\alpha_j(x)\}_{j\ge0}$ and a subsequence
$\{T^{n_j}\}_{j\ge0}$ of $\{T^n\}_{n\ge0}$ such that
$$
\alpha_j(x)\kern1ptT^{n_j}y\wconv x.
$$
Set ${\wtil y=Jy}.$ Every $f$ in ${\X_\vphi}^{\!\!\!*}=\B[\X_\vphi,\CC]$ is
of the form ${f=\wtil fJ}$ for some $\wtil f$ in $\wtil\X^*=\B[\wtil\X,\CC]$,
and every $\wtil x$ in $\wtil\X$ is of the form $\wtil x=Jx$ for some $x$ in
$\X_\vphi.$ Then for each ${\wtil x\in\wtil\X}$ set
$\wtil\alpha_j(\wtil x)=\wtil\alpha_j(Jx)=\alpha_j(x)$ with ${x\in\X_\vphi}.$
Take an arbitrary ${\wtil f\in\wtil\X^*}\!.$ Thus for each
${\wtil x\in\wtil\X}$ there is a sequence $\{\wtil\alpha_j(\wtil x)\}_{j\ge0}$
(independent of $\wtil f$) such that
\begin{eqnarray*}
\wtil f(\wtil\alpha_j(\wtil x)\kern1pt\wtil T^{n_j}\wtil y-\wtil x)
&\kern-6pt=\kern-6pt&
\wtil f(\alpha_j(x)\kern1ptJ\kern1ptT^{n_j}y-Jx)                          \\
&\kern-6pt=\kern-6pt&
\wtil fJ(\alpha_j(x)\kern1ptT^{n_j}y-x)
=f(\alpha_j(x)T^{n_j}y-x)
\end{eqnarray*}
with ${x\in\X_\vphi}$ and ${f\in{\X_\vphi}^{\!\!\!*}}.$ So if ${y\in\X_\vphi}$
is a weakly l-sequentially supercyclic vector for ${T\in\B[\X_\vphi]}$ such
that, for each ${x\in\X_\vphi}$, $\,{f(\alpha_j(x)T^{n_j}y-x)\to0}$ for every
${f\in{\X_\vphi}^{\!\!\!*}}$, then, for each ${\wtil x\in\wtil\X}$,
$\,{\wtil f(\wtil\alpha_j(\wtil x)\kern1pt\wtil T^{n_j}\wtil y-\wtil x)\to0}$
for every ${\wtil f\in\wtil\X}.$ Hence ${\wtil y=Jy}$ is a weakly
l-sequentially supercyclic vector for ${\wtil T\in\B[\wtil\X]}$, proving (1):
for every ${\wtil x\in\wtil\X}$
$$
\wtil \alpha_j\kern1pt(\wtil x)\kern1pt\wtil T^{n_j}\wtil y\wconv
\wtil x.
$$
\vskip-2pt

\vskip6pt\noi
{\it Proof of\/ \rm(2)}\/.
Take an arbitrary ${\what x\in\what\X}.$ Since ${\wtil\X^-\!=\what\X}$, there
is a sequence $\{\wtil x_k\}_{k\ge0}$, with
${\wtil x_k=\wtil x_k(\what x)\in\wtil\X}$ for each $k$, such that
${\wtil x_k\to\what x}.$ Then by the above convergence
$$
\wtil\alpha_j\kern1pt(\wtil x_k)\kern1pt\wtil T^{n_j}\wtil y
\jwconv\wtil x_k\kconv\what x.                                      \eqno(*)
$$
{\it This ensures the existence of a sequence of nonzero numbers
$\{\what\alpha_i(\what x)\}_{i\ge1}$ such that
\goodbreak\vskip4pt\noi
$$
\what\alpha_i(\what x)\kern1pt\what T^{n_i}\wtil y\wconv\what x    \eqno(**)
$$
for a subsequence}\/ $\{\what T^{n_i}\}_{i\ge1}$ of $\{\what T^n\}_{n\ge0}.$
\vskip6pt\noi
Indeed, consider both convergences in ($*$)$.$ Take an arbitrary ${\veps>0}.$
Thus by the second convergence in $(*)$ there exists a positive integer
$k_\veps$ such that ${\|\wtil x_k-\what x\|_{\what\X}\le\frac{\veps}{2}}$
whenever ${k\ge k_\veps}.$ Recall that every $\wtil f$ in $\wtil\X^*$ is of
the form $\wtil f=\what f|_{\wtil\X}$ for some $\what f$ in
$\what\X^*\!=\B[\what\X,\CC].$ Take an arbitrary $\what f$ in $\what\X^*$ with
${\|\what f\|=1}.$ By the first convergence in $(*)$, for each $k$ there
exists a positive integer $j_{\kern1pt\veps,k}$ such that
${|\wtil f(\wtil\alpha_j(\wtil x_k)\kern1pt\wtil T^{n_j}\wtil y-\wtil x_k)|
\le\frac{\veps}{2}}$
for $\wtil f=\what f|_{\wtil\X}$ for every ${\what f}$ with ${\|\what f\|=1}$
whenever ${j\ge j_{\veps,k}}.$ However, by taking ${k=k_\veps}$,
\begin{eqnarray*}
\big|\what f(\wtil\alpha_j(\wtil x_{k_\veps})\kern1pt\what T^{n_j}\wtil y
-\what x)\big|
&\kern-6pt=\kern-6pt&
\big|\what f|_{\wtil\X}(\wtil\alpha_j(\wtil x_{k_\veps})\kern1pt
(\what T|_{\wtil\X})^{n_j}\wtil y-\wtil x_{k_\veps})
+\what f(\wtil x_{k_\veps}-\what x)\big|                                  \\
&\kern-6pt\le\kern-6pt&
\big|\wtil f(\wtil\alpha_j(\wtil x_{k_\veps})\kern1pt\wtil T^{n_j}\wtil y
-\wtil x_{k_\veps})\big|+\|\wtil x_{k_\veps}\!-\what x\|_{\what\X}\le\veps
\end{eqnarray*}
whenever ${j\ge j_{\veps,k_\veps}}.$ Take an arbitrary integer ${i\ge1}$ and
set ${\veps=\frac{1}{i}}.$ Consequently, set $k(i)={k_\veps=k_{\frac{1}{i}}}$
and ${j(i)=j_{\veps,k_\veps}\!=j_{\frac{1}{i},k_i}}\!$ for every ${i\ge1}.$
So for every integer ${i\ge1}$ there is an integer ${j(i)\ge0}$ such
that, setting
$\wtil\alpha_j(\wtil x_{k(i)})
=\wtil\alpha_j(\wtil x_{k_\veps})
=\wtil\alpha_j(\wtil x_{k_\veps}(\what x))$,
we get
${|\what f(\wtil\alpha_j(\wtil x_{k(i)})\kern1pt\what T^{n_j}\wtil y
-\what x)|}\le\smallfrac{1}{i}$
whenever ${j\ge j(i)}.$ Then by taking ${j=j(i)}$,
$$
\big|\what f(\wtil\alpha_{j(i)}(\wtil x_{k(i)})\kern1pt\what T^{n_{j(i)}}
\wtil y-\what x)\big|\le\smallfrac{1}{i}
\quad\;\hbox{for every integer ${i\ge1}$}.
$$
Therefore for each ${\what x\in\what\X}$ there exists a sequence
$\{\wtil\alpha_{j(i)}(\wtil x_{k(i)})\}_{i\ge1}$ (that does not depend on
$\what f$) for which
${\what f(\wtil\alpha_{j(i)}(\wtil x_{k(i)})\kern1pt\what T^{n_{j(i)}}\wtil y
-\what x)\to0}$
for an arbitrary ${\what f\in\X^*}$ with ${\|\what f\|=1}$, and hence
${\what f(\wtil\alpha_{j(i)}(\wtil x_{k(i)})\kern1pt\what T^{n_{j(i)}}\wtil y
-\what x)\to0}$
for every ${\what f\in\X^*}\!$, which means.
$$
\wtil\alpha_{j(i)}
(\wtil x_{k(i)})\kern1pt\what T^{n_{j(i)}}\wtil y\wconv\what x.
$$
For each integer ${i\ge1}$ set
${\what\alpha_i(\what x)=\wtil\alpha_{j(i)}(\wtil x_{k(i)})
=\wtil\alpha_{j(i)}(\wtil x_{k(i)}(\what x))}$ and
${\what T^{n_i}\kern-1pt=\what T^{n_{j(i)}}}.$ Thus for each
${\what x\in\what\X}$ there is a sequence $\{\what\alpha_i(\what x)\}_{i\ge1}$
and a subsequence $\{\what T^{n_i}\}_{i\ge1}$ of $\{\what T^n\}_{n\ge0}$ such
that ($**$) holds$.$ This proves (2) with ${\wtil y\in\wtil\X\dsse\what\X}$:
for every ${\what x\in\what\X}$
$$
\what\alpha_i\kern1pt(\what x)\kern1pt\what T^{n_i}\wtil y\wconv\what x.
$$
So $\wtil y$ is a weakly l-sequentially supercyclic vector for
${\what T\in\B[\what\X]}$, proving Claim 1$.\!\!\!\qed$

\vskip6pt\noi
{\it Note}\/$.$
A norm topology version of Claim 1 (i.e., a version of Claim 1 for
supercyclicity) is known since Ansari--Bourdon's paper [1] although as far as
we are aware such a supercyclic case has been left without a proof so far
(see, e.g., [1, p.198])$.$ The weak counterpart proved in Claim 1 above
(i.e., the version for weak l-sequential supercyclicity) has a natural
transcription for the supercyclic case, thus supporting a proof (using the
same argument) for the strong (norm topology) version as well.

\vskip6pt\noi
Then by (a), (b), and Claim 1 we get (c)$.$ Indeed, since $T$ is an isometry
on $\X_\vphi$, then so is its extension $\what T$ on the completion $\what\X$
of $\X_\vphi.$ Since $\R(\what T)$ is dense by supercyclicity (Remark 2.1),
and closed as $\what T$ is a linear isometry on a Banach space, then $\what T$
is a surjective linear isometry, which means an isometric isomorphism.

\vskip6pt\noi
(d)
Take any ${x\in\X}.$ Consider the setup in the proof of items (b,c)$.$ Since
${\X^*_\vphi\sse\X^*}\!$, if ${f(T^nx)\to0}$ for every ${f\in\X^*}$, then in
particular ${f(T^nx)\to0}$ for every ${f\in\X^*_\vphi}$:
$$
\hbox
{if $T$ on $\X$ is weakly stable, then $T$ on $\X_\vphi$ is weakly stable.}
                                                                \eqno{\rm(i)}
$$
The same argument and notation as in proof of (c), where
${J\!:\X_\vphi\kern-1pt\to\wtil\X\kern-1pt\dsse\kern-1pt\what\X}$ is an
isometric isomorphism, lead to the identity (with ${\wtil x=Jx}$ and
$\wtil T=\what T|_{\wtil\X}=J\kern1ptTJ^{-1}$)
\goodbreak\vskip4pt\noi
$$
f(T^nx)
=\wtil fJ(T^nx)
=\wtil f(J\kern1ptT^nx)
=\wtil f(\wtil T^n\wtil x)
=\what f|_{\wtil\X}(\what T|_{\wtil\X})^n\wtil x)
=\what f(\what T^n\wtil x)
$$
for an arbitrary pair of vectors
$({x=J^{-1}\wtil x\in\X_\vphi}$,
$\,{\wtil x=Jx\in\wtil\X})$,
and an arbitrary triple of functionals
$({f\kern-1pt\in\kern-1pt\X_\vphi^*}$,
$\;{\wtil f\!=\kern-1ptfJ^{-1}
\!=\what f|_{\wtil\X}\kern-1pt\in\kern-1pt\wtil\X^*}\!$,
$\;{\what f\kern-1pt\in\kern-1pt\what\X^*}$
--- the extension by continuity of $\wtil f$ from $\wtil\X$ over
$\what\X=\wtil\X^-)$, where ${J\kern1ptT^n=\wtil T^nJ=\what T^nJ}$ for every
$n$ (as $\wtil\X$ is $\what T$-in\-variant), and
$\sup_n\|\what T^n\|=\sup_n\|T^n\|=\beta.$ Take an arbitrary ${\veps>0}.$
Since ${\wtil\X^-\!=\what\X}$, for each ${\what x\in\what\X}$ there is an
${x=J^{-1}\wtil x\in\X_\vphi}$ such that
${\|\what x-\wtil x\|_{\what\X}}<\veps.$ Also
${f(T^nx)=\what f(\what T^n\wtil x)}$ as seen above$.$ Hence for each
${\what x\in\what\X}$ there is an ${x\in\kern-1pt\X_\vphi}$ such that
$$
\big|\,|f(T^nx)|-|\what f(\what T^n\what x)|\,\big|
=\big|\,|\what f(\what T^n\wtil x)|-|\what f(\what T^n\what x)|\,\big|
<\|\what f\|\,\beta\,\veps,
$$
where $\|f\|=\|\wtil f\|=\|\what f\|.$ Thus if ${f(T^nx)\to0}$ for every
${f\kern-1pt\in\kern-1pt\X_\vphi^*}$, then ${\what f(\what T^n\what x)\to0}$
for every ${\what f\kern-1pt\in\kern-1pt\what\X^*}$ (the converse is
trivial)$.$ Therefore
$$
\hbox{$T$ on $\X_\vphi$ is weakly stable if and only if
$\what T$ on $\what\X$ is weakly stable}.                     \eqno{\rm(ii)}
$$
By (i) and (ii) we get the result in (d).
\end{proof}

\vskip0pt\noi
\begin{remark}
(a) {\it There is no weakly l-sequentially supercyclic nonsurjective
isometry}\/$.$ In other words, every weakly l-sequentially supercyclic
isometry on a Banach space is surjective, thus an isometric isomorphism$.$
Indeed, isometries on a Banach space have closed range (norm topology, since
isometries are bounded below), and any form of supercyclicity leads to a dense
range (any topology, Remark 2.1)$.$ Thus a nonsurjective isometry is not
weakly supercyclic, and so it is not weakly l-sequentially supercyclic.

\vskip6pt\noi
(b) {\it There is no weakly l-sequentially supercyclic compact operator of
class}\/ $C_{1{\textstyle\cdot}}.$ In fact, a weakly l-sequentially
supercyclic compact operator $T$ is quasinilpotent (i.e., ${r(T)=0}$)
\cite[Theorem 4.2]{KD2}, and so uniformly stable, thus of class
$C_{00}$.
\end{remark}

\vskip4pt
If the norm ${\|\cdot\|}$ of $\X$ in Theorem 3.1 is induced by an inner
product ${\<\cdot\,;\cdot\>}$, then the norm $\|\cdot\|_{\what\X}$ is also
induced by an inner product and so $\what\X$ is a Hilbert space$.$ In this
case a Hilbert space version of Theorem 3.1 can be stated, where the new norm
that makes $T$ into an isometry is now given by
$\|x\|_\vphi^2=\vphi(\{\<T^nx\,;T^nx\>\})$ for each ${x\in\X}$, and an
isometric isomorphism now means a unitary transformation.

\vskip3pt\noi
\begin{theorem}
Let\/ $T$ be a power bounded operator of class\/ $C_{1{\textstyle\cdot}\!}$
on an inner product space\/ ${(\X,\<\cdot\,;\cdot\>)}.$ Then
\vskip4pt\noi
\begin{description}
\item{$\kern-7pt$\rm(a)$\kern3pt$}
there is an inner product ${\<\cdot\,;\cdot\>_\vphi}\kern-1pt$ on $\X\!$
for which $T\!$ is an isometry on ${(\X,\<\cdot\,;\cdot\>_{\vphi})}$.
\vskip4pt\noi
\item{$\kern-8pt$\rm(b)$\kern2pt$}
If a vector\/ ${y\in\X}$ is weakly l-sequentially supercyclic for\/ $T$
when it acts on the inner product space\/ ${(\X,\<\cdot\,;\cdot\>)}$, then
the same vector\/ ${y\in\X}$ is weakly l-se\-quentially supercyclic for\/
$T\kern-1pt$ when it acts on the inner product space\/
${(\X,\<\cdot\,;\cdot\>_{\vphi})}$.
\vskip4pt\noi
\item{$\kern-7pt$\rm(c)$\kern3pt$}
If\/ $T$ is weakly l-sequentially supercyclic when acting on\/
${(\X,\<\cdot\,;\cdot\>)}$, then it has an extension\/ $\what T$ on the
completion\/ ${(\what\X,\<\cdot\,;\cdot\>_{\what\X})}$ of\/
${(\X,\<\cdot\,;\cdot\>_\vphi)}$ which is a
weakly l-sequentially supercyclic unitary transformation\/.
\vskip4pt\noi
\item{$\kern-8pt$\rm(d)$\kern2pt$}
If\/ $T$ on\/ ${(\X,\<\cdot\,;\cdot\>)}$ is weakly stable, then\/
$\what T$ on\/ ${(\what\X,\<\cdot\,;\cdot\>_{\what\X})}$ is weakly stable.
\end{description}
\end{theorem}

\begin{proof}
Essentially the same argument as in proof of Theorem 3.1.

\vskip6pt\noi
Let ${\vphi\!:\ell_+^\infty\!\to\kern-1pt\CC}$ be a Banach limit and let
${\|\cdot\|\!:\X\!\to\kern-1pt\RR}$ be the norm induced by the inner product
${\<\cdot\,;\cdot\>\!:\X\times\X\!\to\CC}.$ Suppose ${\!T\in\kern-1pt\BX}$
is power bounded$.$ Thus set
\goodbreak\vskip4pt\noi
$$
\<x\,;z\>_\vphi\!=\vphi(\{\<T^nx\,;T^nz\>\})
$$
for each ${x,z}$ in $\X.$ Since $\vphi$ is linear (which implies
$\vphi(\overline{\{\xi_n\}})=\overline\vphi(\{\xi_n\})$) and positive
(i.e., ${0\le\vphi(\{\xi_n\})}$ whenever ${0\le\xi_n}$ for every $n$), and
since $T$ is linear, then it is readily verified that
${\<\cdot\,;\cdot\>_\vphi\!:\X\times\X\!\to\CC}$ is a semi-inner product on
$\X.$ Hence
$$
\|x\|_\vphi^2=\vphi(\{\|T^nx\|^2\})
$$
for every ${x\in\X}$ where ${\|\cdot\|_\vphi\!:\X\!\to\RR}$ is the seminorm
induced by the semi-inner product
${\<\cdot\,;\cdot\>_\vphi}$ so that
${\|x\|_\vphi\le\sup_n\|T^n\|\|x\|}.$ (Even in this case of norms of a power
sequence of a power bounded operator, the squares in the above identity cannot
be omitted due to the nonmultiplicativity of Banach limits)$.$ From now on the
proof develops as the proof for Theorem 3.1, where the previous seminorm is
replaced by the above one, which becomes a norm under the same assumption of
$T$ being of class $C_{1{\textstyle\cdot}}.$ As before we work with the norms
${\|\cdot\|}$ and ${\|\cdot\|_\vphi}$ on $\X$ regardless the fact that they
may have been induced by inner products, and again write $\X$ for the inner
product space ${(\X,\<\cdot,\cdot\>)}$ and $\X_\vphi$ for the inner product
space ${(\X,\<\cdot,\cdot\>_\vphi)}$.
\end{proof}

\vskip0pt\noi
\begin{remark}
Let $T$ be a weakly stable (thus power bounded) weakly l-sequentially
supercyclic operator of class $C_{1{\textstyle\cdot}}$ on an inner product
space$.$ Consider Theorem 3.2$.$ Since $\what T$ is an isometry, it is not
supercyclic \cite[Theorem 2.1]{AB}$.$ Since $\what T$ is a weakly
l-sequentially supercyclic and weakly stable unitary operator, then it is
singular-continuous \cite[Theorem 4.2]{Kub1} and its scalar spectral measure
is Rajchman$.$ Recall$:$ a unitary operator on a Hilbert space is weakly
stable if and only if its scalar spectral measure is a Rajchman measure (i.e.,
a measure $\mu$ on the $\sigma$-algebra of Borel subsets of the unit circle
$\TT$ for which ${\int_\TT\lambda^k\,d\mu\to0}$ for${|k|\to\infty}$)
\cite[Proposition 3.3]{Kub1}.
\vskip6pt\noi
{\narrower
{\it If\/ $T$ is a weakly stable weakly l-sequentially supercyclic power
bounded operator of class\/ $C_{1{\textstyle\cdot}\!\!}$ on an inner product
space, then it has an extension\/ $\what T$ on a Hilbert space which is a
weakly stable weakly l-sequentially supercyclic\/ $($but not supercyclic\/$)$
singular-continuous unitary operator\/ $($thus of class\/ $C_{11})$ whose
scalar spectral measure is a Rajchman measure}\/.
\vskip0pt}\noi
\end{remark}

\section{Two Applications}

Let $\sigma_{\kern-1ptP}(\,\cdot\,)$ stand for point spectrum$.$ Let
the continuous linear operator $T^*$ stand for the normed-space (or
topological-vector-space) adjoint of a continuous linear operator $T$ (if
$T$ is a Hilbert-space operator, then $T^*$ is identified with the
Hilbert-space adjoint of $T.)$ If $T$ is supercyclic, then
${\#\sigma_{\kern-1ptP}(T^*)\le1}$ (i.e., then the cardinality of the point
spectrum of $T^*$ is not greater than one)$.$ In other words, if $T$ is
supercyclic, then the adjoint of $T$ has at most one eigenvalue$.$ This has
been verified for supercyclic operators in a Hilbert-space setting in
\cite[Proposition 3.1]{Her}, extended to operators on a normed space in
\cite[Theorem 3.2]{AB}), and further extended to supercyclic operators on a
locally convex space in \cite[Lemma 1, Theorem 4]{Per}$.$ But the weak
topology of a normed space is a locally convex subtopology of the locally
convex norm topology (see, e.g., \cite[Theorems 2.5.2 and 2.2.3]{Meg})$.$ Then
the latter extension holds in particular on a normed space under the weak
topology, thus including weakly supercyclic operators and consequently weakly
l-sequentially supercyclic operators on normed spaces:
$$
\hbox{If $T$ is weakly l-sequentially supercyclic,
then ${\#\sigma_{\kern-1ptP}(T^*)\le1}$}.
$$
We show next that ${\#\sigma_{\kern-1ptP}(T^*)=\#\sigma_{\kern-1ptP}(T)=0}$
if the weakly l-sequentially super\-cyclic power bounded operator $T$ is of
class $C_{1{\textstyle\cdot}}.$ Therefore ${\#\sigma_{\kern-1ptP}(T^*)=1}$
only if ${T^nx\to0}$ for some nonzero ${x\in\X}$.

\vskip3pt\noi
\begin{corollary}
Let\/ $T$ be a power bounded operator of class\/ $C_{1{\textstyle\cdot}\!}$
on a Hilbert space\/ $\X.$ If\/ $T$ is weakly l-sequentially supercyclic, then
$$
\sigma_{\kern-1ptP}(T)=\sigma_{\kern-1ptP}(T^*)=\void.
$$
\end{corollary}

\vskip0pt\noi
{\it Proof}\/$.$
Let $T$ be a power bounded operator of class $C_{1{\textstyle\cdot}}$ on a
normed space $\X.$ Consider the extension
${\what T\kern-1pt\in\kern-1pt\B[\what\X]}$ of
${T\kern-1pt\in\kern-1pt\B[\X_\vphi]}$ over the completion $\what \X$ of
$\X_\vphi$ as in the proof of Theorem 3.1$.$ We split this proof into three
parts.

\vskip6pt\noi
{\sc Part 1}$.$
Clearly $\sigma_{\kern-1ptP}(T)$ is the same regardless whether $T$ acts on
$\X={(\X,\|\cdot\|)}$ or in $\X_\vphi={(\X,\|\cdot\|_\vphi)}$ since point
spectrum is purely an algebraic notion and in both cases the linear $T$ acts
on the same linear space $\X.$ Let $T$ act on $\X_\vphi.$ Thus
$T=$ ${J^{-1}\what T|_{\wtil\X}J\in\B[\X_\vphi]}$ where
${J\in\B[\X_\vphi,\wtil\X]}$ is an isometric isomorphism and
$\wtil\X$ is $\what T$-in\-variant (cf$.$ proof of Theorem 3.1)$.$ Since
similarity preserves the spectrum and its parts, then
$\sigma_{\kern-1ptP}(T)=\sigma_{\kern-1ptP}(\what T|_{\wtil\X})
\sse\sigma_{\kern-1ptP}(\what T).$
Therefore for $T$ acting on $\X$ or on $\X_\vphi$
$$
\sigma_{\kern-1ptP}(T)\ne\void
\quad\limply\quad
\sigma_{\kern-1ptP}(\what T)\ne\void.
$$

\vskip4pt\noi
{\sc Part 2}$.$
The definition of adjoint involves topology and although the duals are nested
they may not coincide$.$ Let $T^*$ in $\B[\X^*]$ be the adjoint of $T$ acting
on $\X={(\X,\|\cdot\|)}.$ Use the same notation $T^*$ for the adjoint in
$\B[\X^*_\vphi]$ of $T$ acting on $\X_\vphi={(\X,\|\cdot\|_\vphi)}).$ First
suppose there exits ${\lambda\in\sigma_{\kern-1ptP}(T^*)}$ for $T$ acting on
$\X_\vphi$ which means $\lambda\kern1ptg=T^*\kern-1ptg$ for some nonzero
${g\in\X_\vphi^*}.$ Then ${g\in\X^*}$ and so
${\lambda\in\sigma_{\kern-1ptP}(T^*)}$ for $T$ acting on $\X.$ Thus the
inclusion ${\X_\vphi^*\sse\X^*}$ (which may be proper) allows us to infer
$$
\sigma_{\kern-1ptP}(T^*)\ne\void
\;\;\hbox{for $T$ acting on $\X_\vphi$}
\quad\limply\quad
\sigma_{\kern-1ptP}(T^*)\ne\void
\;\;\hbox{for $T$ acting on $\X$}.
$$
The converse requires a different argument$.$ Suppose $\X$ is a Hilbert
space and let ${T^*\!\in\BX}$ be the Hilbert-space adjoint of ${T\in\BX}$.

\vskip6pt\noi
{\it Claim 2}\/$.$
There exists a positive operator ${A\in\BX}$ (i.e., ${A>O}$) for which
$$
\<x\,;z\>_\vphi=\<Ax\,;z\>
\;\;\hbox{for every $x,z\in\X$,
\quad with $\;\R(A)^-\!=\X$,
\quad and
\quad $T^*\!A\kern1ptT=A$}.
$$

\vskip4pt\noi
{\it Proof}\/$.$
Since the inner product ${\<\cdot\,;\cdot\>_\vphi\!:\X\times\X\!\to\CC}$ is a
bounded sesquilinear form (i.e.,
$|\<x\,;z\>_\vphi|
=|\vphi(\{\<T^nx,T^nz\>\})|
\le\|\vphi\|\sup_n|\<T^nx\,;T^nz\>|
\le\beta^2\|x\|\,\|z\|$
where $\|\vphi\|=1$ and $\beta=\sup_n\|T^n\|$), then a classical result from
\cite[Theorem 2.28]{Sto} ensures the existence of such a positive operator
$A.$ Since ${A>O}$ is injective and self-adjoint, then the range $\R(A)$ of
$A$ is dense in ${(\X,\<\cdot\,;\cdot\>)}.$
Moreover since $\|Tx\|_\vphi=\|x\|_\vphi$,
$$
\<Ax\,;x\>=\<x\,;x\>_\vphi
=\<Tx\,;Tx\>_\vphi
=\<A\kern1ptTx\,;Tx\>
=\<T^*\!A\kern1ptTx\,;x\>
$$
for every ${x\in\X}$ and so, as ${A-T^*\!A\kern1ptT}$ is self-adjoint,
${A=T^*\!A\kern1ptT}.\!\!\!\qed$

\vskip6pt\noi
Take ${f\in\X^*}\!.$ Since $\X$ is a Hilbert space, ${f=\<\,\cdot\,;y\>}$ for
some ${y\in\X}.$ Take the restriction ${f|_{\R(A)}\!:\R(A)\sse\X\to\CC}.$
Since $fA={\<A\,\cdot\,;y\>=\<\,\cdot\,;y\>_\vphi\in\X^*_\vphi}$ we can
conclude that ${f|_{\R(A)}\in\X^*_\vphi}.$ Now take any ${g\in\R(A)^*}\!.$
Since ${\R(A)^-\!=\X}$, then by continuity there exists ${\what g\in\X^*}$
such that $g=\what g|_{\R(A)}$, which is in $\X^*_\vphi$ as we saw above$.$
So ${\R(A)^*\sse\X^*_\vphi}.$ Moreover, since ${\R(A)^-\!=\X}$, then
$\X^*\cong\R(A)^*$ where $\cong$ denotes isometrically isomorphic (see, e.g.,
\cite[Exercise 1.112, p.95]{Meg})$.$ Therefore
$$
\R(A)^-\!=\X
\quad\limply\quad
\X^*\cong\R(A)^*\sse\X^*_\vphi.
$$
Take an isometric isomorphism
${W\!:\X^*\!\to W(\X^*)=\R(A)^*\!\sse\X^*_\vphi}.$ If $\lambda$ lies in
${\sigma_{\kern-1ptP}(T^*)}$ for $T$ acting on $\X$, then there is a nonzero
$f$ in $\X^*$ for which ${\lambda f=T^*f}.$ Set $g=Wf$ in
$\X_\vphi^*.$ By definition of adjoint,
$\lambda\,g
=\lambda\,Wf
=W\lambda\kern1ptf=
W\kern1ptT^*f
=Wf\kern1ptT
=g\,T
=T^*g$
(same notation for $T^*$ on $\X^*$ or $\X^*_\vphi).$ Then
${\lambda\in\sigma_{\kern-1ptP}(T^*)}$ for $T$ acting on $\X_\vphi.$ Hence
$$
\sigma_{\kern-1ptP}(T^*)\ne\void
\;\;\hbox{for $T$ acting on $\X$}
\quad\limply\quad
\sigma_{\kern-1ptP}(T^*)\ne\void
\;\;\hbox{for $T$ acting on $\X_\vphi$}.                    \eqno{(\rm iii)}
$$
Next let $T$ act on $\X_\vphi$ and take its adjoint $T^*$ on $\X^*_\vphi.$
Since $T=J^{-1}\what T_{\wtil\X}J\in\B[\X_\vphi]$ we get
$T^* f
=f\kern1ptT
=fJ^{-1}\what T|_{\wtil\X}J
=(J^{-1}\what T|_{\wtil\X}J)^*f
=J^*(\what T|_{\wtil\X})^*J^{*-1}f$
for every ${f\in\X_\vphi^*}.$ Thus $T^*=J^*(\what T|_{\wtil\X})^*J^{*-1}$
(i.e., $T^*$ and $(\what T|_{\wtil\X})^*$ are isometrically isomorphic) and
so ${\sigma_{\kern-1ptP}(T^*)}=\sigma_{\kern-1ptP}((\what T|_{\wtil\X})^*).$
Also, ${\wtil\X^*\cong\what\X^*}$ as ${\wtil\X^-\!\!=\what\X}$ (see, e.g.,
\cite[Exercise 1.112]{Meg} again)$.$ Take an isometric isomorphism
${K\!:\wtil\X^*\!\to\what\X^*}\!$ and an arbitrary $\wtil f$ in $\wtil\X^*\!.$
By continuity $\wtil f=\what f|_{\wtil\X}=K^{-1}\what f$ for a unique
${\what f\in\what\X^*}.$ Then
$K(\what T|_{\wtil\X})^*\wtil f
=K\wtil f(\what T|_{\wtil\X})
=\what f\,\what T
=\what T^*\what f
=\what T^*K\wtil f.$
Hence $K(\what T|_{\wtil\X})^*\!=\what T^*K$ (i.e., $(\what T|_{\wtil\X})^*$
and $\what T^*$ are isometrically isomorphic)$.$ So
$\sigma_{\kern-1ptP}((\what T|_{\wtil\X})^*)=\sigma_{\kern-1ptP}(\what T^*).$
Thus
$\sigma_{\kern-1ptP}(T^*)
=\sigma_{\kern-1ptP}((\what T|_{\wtil\X})^*)
=\sigma_{\kern-1ptP}(\what T^*).$
Then
$$
\sigma_{\kern-1ptP}(T^*)\ne\void
\;\;\hbox{for $T$ acting on $\X_\vphi$}
\quad\limply\quad
\sigma_{\kern-1ptP}(\what T^*)\ne\void.                   \eqno{(\rm iv)}
$$
According to (iii) and (iv) we get
$$
\sigma_{\kern-1ptP}(T^*)\ne\void
\;\;\hbox{for $T$ acting on $\X$}
\quad\limply\quad
\sigma_{\kern-1ptP}(\what T^*)\ne\void.
$$

\vskip4pt\noi
{\sc Part 3}$.$
By Theorem 3.2 $\what T$ is a weakly l-sequentially supercyclic unitary
operator on the Hilbert space $\what\X.$ Then $\what T^*$ also is a weakly
l-sequentially supercyclic unitary on $\what\X$ \cite[Theorem 3.3]{KD2}$.$ But
weakly l-sequentially supercyclic unitary operators are singular-continuous
\cite[Theorem 4.2]{Kub1}, and so have no eigenvalues$:$
$\sigma_{\kern-1ptP}(\what T)=\sigma_{\kern-1ptP}(\what T^*)=\void.$
Therefore by Parts 1 and 2, for $T$ acting on the Hilbert space $\X$,
$$
\sigma_{\kern-1ptP}(T)=\sigma_{\kern-1ptP}(T^*)
=\sigma_{\kern-1ptP}(\what T)=\sigma_{\kern-1ptP}(\what T^*)
=\void. \eqno{\qed}
$$

\vskip0pt\noi
\begin{corollary}
Let\/ ${T\in\BX}$ be a power bounded operator of class\/
$C_{1{\textstyle\cdot}\!}$ acting on a Hilbert space\/ $\X.$ Let\/
${\|\cdot\|}$ be the norm generated by the inner product\/
${\<\cdot\,;\cdot\>}$ on\/ $\X.$ Consider the new inner product space\/
$\X_\vphi$ where\/ ${\|\cdot\|_\vphi}$ is the new norm induced by the new
inner product\/ ${\<\cdot\,;\cdot\>_\vphi}$ on\/ $\X$ as in Theorem 3.2$.$
Take the positive operator\/ ${A\in\BX}$ for which\/
${\<\cdot\,;\cdot\>_\vphi}={\<A\,\cdot\,;\cdot\>}$ as in Claim 2\/
$($proof of Corollary 4.1\/$).$ The following assertions are pairwise
equivalent.
\vskip4pt\noi
\begin{description}
\item{$\kern-4pt$\rm(a)$\kern1pt$}
$A$ is invertible \quad$($i.e., has bounded inverse on\/ $\X)$.
\vskip4pt\noi
\item{$\kern-5pt$\rm(b)$\kern0pt$}
The norms\/ ${\|\cdot\|}$ and\/ ${\|\cdot\|_\vphi}$ on\/ $\X$ are equivalent.
\vskip4pt\noi
\item{$\kern-4pt$\rm(c)$\kern2pt$}
$T$ acting on\/ $\X$ is similar to an isometry.
\vskip4pt\noi
\item{$\kern-4pt$\rm(d)$\kern0pt$}
$T$ acting on\/ $\X$ is power bounded below \quad$($i.e., there is a
constant\/ ${\gamma>0}$ such that\/ ${\gamma\|x\|\le\|T^nx\|}$ for all
integers ${n\ge1}$ and every\/ ${x\in\X})$.
\end{description}
\vskip4pt\noi
If\/ $T$ is weakly l-sequentially supercyclic, then any of the
above equivalent assertions implies\/ $T$ is unitary when acting on\/
$\X_\vphi$, which in turn implies
\vskip4pt\noi
\begin{description}
\item{$\kern-4pt$\rm(e)$\kern2pt$}
$T$ acting on\/ $\X$ is similar to a unitary operator.
\end{description}
\end{corollary}

\begin{proof}
(a)$\kern-1pt\iff$\kern-1pt(b)$:$
Since $A$ is self-adjoint and injective it has a dense range$.$ If in
addition $A$ is bounded below, then it has a closed range and so it is
surjective$.$ Thus the positive operator $A$ is invertible (i.e., it has a
bounded inverse on $\X$, equiva\-lently, it is bounded below and surjective)
if and only if (cf$.$ proof of Claim 2)
$\gamma\|x\|
\le\|A^{\frac{1}{2}}x\|\!
=\|x\|_\vphi
\le\beta\|x\|$
for every $x$ in $\X$ for some ${\gamma>0}$ (and in this case
$\gamma=\|A^{-\frac{1}{2}}\|^{-1})$, which means the norms ${\|\cdot\|}$ and
${\|\cdot\|_\vphi}$ are equivalent.

\vskip6pt\noi
(a)$\,\limply$(c)$\,\limply$(b)$:$
In fact,
$\|A^\frac{1}{2}Tx\|^2
=\<{A^\frac{1}{2}Tx\,;A^\frac{1}{2}Tx}\>
=\<{T^*\!A\kern1ptTx\,;x}\>
=\<{Ax\,;x}\>
=\|A^\frac{1}{2}x\|^2$
for every ${x\in\X}$ since ${T^*\!A\kern1ptT=A}.$ Thus if (a) holds,
then $\|A^\frac{1}{2}TA^{-\frac{1}{2}}x\|=\|x\|$ for every ${x\in\X}$, and so
$T$ is similar to an isometry (i.e., (c) holds)$.$ Conversely if (c) holds,
then there is an operator ${S\in\BX}$ for which $S^{-1}T^nS$ is an isometry
for every ${n\ge0}$ so that the sequence $\{\|S^{-1}T^nSx\|^2\}$ is
constantly equal to $\|x\|^2\!$, and hence
$\|x\|_\vphi^2
\le$
$\beta^2\|x\|^2
=\beta^2\vphi(\{\|S^{-1}T^nSx\|^2\})
\le\beta^2\|S^{-1}\|^2\vphi(\{\|T^nSx\|^2\})
=\beta^2\|S^{-1}\|^2\|Sx\|_\vphi^2
\le\beta^2\|S^{-1}\|^2\|S\|^2\|x\|_\vphi^2$
for every ${x\in\X}.$ Thus (b) holds.

\vskip6pt\noi
(c)$\kern-1pt\iff\kern-1pt$(d)$:$
This is a direct consequence of \cite[Theorem 2]{KR} which reads as follows$.$
{\it An operator\/ $F$ on a normed space\/ $\Z$ is power bounded and power
bounded below\/ $($i.e., there are constants\/ ${\gamma,\beta>0}$ such that\/
$\gamma\|z\|\le\|F^nz\|\le\beta\|z\|$ for all\/ integers\/ ${n\ge1}$ and
every\/ ${z\in\Z})$ if and only if\/ $F$ is similar to an isometry}\/.

\vskip6pt\noi
Moreover if (b) holds, then $\X_\vphi$ is a Hilbert space (because $\X$ is a
Hilbert space)$.$ Then ${\X_\vphi=\wtil\X=\what\X}$ and $T$ on $\X_\vphi$
coincides with $\what T$ on $\what\X$ and, since $T$ on $\X_\vphi$ is weak\-ly
l-sequentially supercyclic whenever $T$ on $\X$ is, $T$ on $\X_\vphi$ is
unitary according to Theorem 3.2$.$ Thus $T$ on $\X_\vphi$ is surjective,
and so is $T$ on $\X.$ Thus since $T$ is surjec\-tive and similar to an
isometry by (c), then the isometry is surjective which means it is unitary$.$
This shows that any of the above equivalent assertions imply (e).
\end{proof}

\vskip0pt\noi
\begin{corollary}
Every weakly l-sequentially supercyclic operator similar to an isometry is
similar to a unitary operator.
\end{corollary}

\begin{proof}
This holds for Hilbert-space operators according to Corollary 4.2(c,e)
since similarity to an isometry trivially implies power boundedness of
class $C_{1{\textstyle\cdot}\!}$.
\end{proof}

\vskip0pt\noi
\begin{remark}
As we saw in Corollary 4.2(c,d) (cf$.$ \cite[Theorem 2]{KR}),
\vskip4pt\noi
{\narrower
\it if an operator on a Hilbert space is power bounded and power bounded
below, then it is similar to an isometry\/.
\vskip4pt}\noi
Moreover, according to Corollary 4.2(c,e),
\vskip4pt\noi
{\narrower
\it if a weakly l-sequentially supercyclic operator on a Hilbert space
is power bounded and power bounded below, then it is similar to a unitary
operator\/.
\vskip4pt}\noi
Indeed, if $T$ is power bounded below, then it is of class
$C_{1{\textstyle\cdot}\!}$ (i.e., $\gamma\|x\|\le\|T^nx\|$ implies
${T^nx\not\to0})$ and the converse fails (see, e.g., \cite[p.69]{MDOT})$.$
Does the converse holds under weak l-sequential supercyclicity
assumption$\kern1pt?$ This suggests the following stronger form of question
(F) (see Section 1) in the sense that an affirmative answer to question
(F$'$) implies an affirmative answer to question (F):
\vskip4pt\noi
\begin{description}
\item{$\kern-9pt$\rm(F$'$)$\kern2pt$}
{\it is a weakly l-sequentially supercyclic power bounded operator of class\/
$C_{1{\textstyle\cdot}\!}$ power bounded below}$\,?$
\end{description}
\end{remark}

\vskip-16pt\noi
\bibliographystyle{amsplain}

\end{document}